\documentclass{amsart}
\usepackage{amssymb,hyperref}

\newcommand{\reg}{\text{\rm reg}}
\newcommand{\ann}{\text{\rm ann}}
\newcommand{\Min}{\text{\rm Min}}

\renewcommand{\ge}{\geqslant}
\renewcommand{\le}{\leqslant}

\newtheorem{theorem}{Theorem}[section]
\newtheorem{lemma}[theorem]{Lemma}
\newtheorem{corollary}[theorem]{Corollary}

\theoremstyle{definition}
\newtheorem{definition}[theorem]{Definition}
\newtheorem{example}[theorem]{Example}

\newtheorem{hypothesis}[theorem]{Hypothesis}
\newtheorem{discussion}[theorem]{Discussion}

\theoremstyle{remark}
\newtheorem{remark}[theorem]{Remark}

\numberwithin{equation}{section}

\title[Asymptotic linear bounds of Castelnuovo-Mumford regularity]{Asymptotic linear bounds of Castelnuovo-Mumford regularity in multigraded modules}

\author{Dipankar Ghosh}
\address{Department of Mathematics, Indian Institute of Technology Bombay, Powai, Mumbai 400076, India}
\email{dipankar@math.iitb.ac.in}


\begin{document}

\begin{abstract}
 Let $A$ be a Noetherian standard $\mathbb{N}$-graded algebra over an Artinian local ring $A_0$. Let
 $I_1,\ldots,I_t$ be homogeneous ideals of $A$ and $M$ a finitely generated $\mathbb{N}$-graded $A$-module. We prove
 that there exist two integers $k, k'$ such that
  \[\reg(I_1^{n_1}\cdots I_t^{n_t} M) \le (n_1 + \cdots + n_t) k + k'
    \quad\mbox{for all }~n_1,\ldots,n_t \in \mathbb{N}.
  \]
\end{abstract}

\maketitle

\section{Introduction}\label{Introduction}
 Let $I$ be a homogeneous ideal of a polynomial ring $S = K[X_1,\ldots,X_d]$ over a field $K$ with usual grading.
 Bertram, Ein and Lazarsfeld \cite{Be91} have initiated the study of the Castelnuovo-Mumford regularity of $I^n$
 as a function of $n$ by proving that if $I$ is the defining ideal of a smooth complex projective variety, then
 $\reg(I^n)$ is bounded by a linear function of $n$.

 Thereafter, Chandler \cite{Ch97} and Geramita, Gimigliano and Pitteloud \cite{Ge95} proved that if
 $\dim(S/I) \le 1$, then $\reg(I^n) \le n\cdot \reg(I)$ for all $n \ge 1$. This result does not hold true for
 higher dimension, due to an example of Sturmfels \cite{St00}. However, in \cite[Theorem~3.6]{Sw97}, Swanson proved
 that $\reg(I^n) \le k n$ for all $n \ge 1$, where $k$ is some integer.
 
 Later, Cutkosky, Herzog and Trung \cite{Cu99}, and Kodiyalam \cite{Ko00} independently proved that $\reg(I^n)$
 can be expressed as a linear function of $n$ for all sufficiently large $n$. Recently, Trung and Wang proved the
 above result in a more general way \cite[Theorem~3.2]{Tr05}; if $S$ is a standard graded ring over a commutative
 Noetherian ring with unity, $I$ a homogeneous ideal of $S$ and $M$ a finitely generated graded $S$-module, then
 $\reg(I^n M)$ is asymptotically a linear function of $n$. 
 
 In this context, the natural question arises ``what happens when we consider several ideals instead of just
 considering one ideal?". More precisely, if $I_1,\ldots,I_t$ are homogeneous ideals of $S$ and $M$ is a finitely
 generated graded $S$-module, then what will be the behaviour of $\reg(I_1^{n_1}\cdots I_t^{n_t} M)$ as a function
 of $(n_1,\ldots,n_t)$?
 
 Let $A = A_0[x_1,\ldots,x_d]$ be a Noetherian standard $\mathbb{N}$-graded algebra over an Artinian local ring
 $(A_0,\mathfrak{m})$. In particular, $A$ can be a coordinate ring of any projective variety over any field with
 usual grading. Let $I_1,\ldots,I_t$ be homogeneous ideals of $A$ and $M$ a finitely generated $\mathbb{N}$-graded
 $A$-module. In this article, we prove that there exist two integers $k, k'$ such that
 \begin{center}
  $(\dagger)$\hfill $\reg(I_1^{n_1}\cdots I_t^{n_t} M) \le (n_1 + \cdots + n_t) k + k'\quad$
  for all ~$n_1,\ldots,n_t \in \mathbb{N}$. \hfill \;
 \end{center}
 
 The rest of the paper is organized as follows. We start by introducing some notations and terminologies
 in Section~\ref{Notation}. In Section~\ref{Preliminaries}, we give some preliminaries on
 Castelnuovo-Mumford regularity and multigraded modules which we use in order to prove our main result. Finally,
 in Section~\ref{Linear bounds of regularity}, we prove $(\dagger)$ in several steps.
 
 \section{Notation}\label{Notation}
 Throughout this article, $\mathbb{N}$ denotes the set of all non-negative integers and $t$ is any fixed positive
 integer. We use small letters with underline (e.g., $\underline{n}$) to denote elements of
 $\mathbb{N}^t$, and we use subscripts mainly to denote the coordinates of such an element, e.g.,
 $\underline{n} = (n_1,n_2,\ldots,n_t)$. In particular, for each $1 \le i \le t$, $\underline{e}^i$ denotes the
 $i^{\rm th}$ standard basis element of $\mathbb{N}^t$. We denote $\underline{0}$ the element of $\mathbb{N}^t$
 with all components $0$. Throughout, we use the partial order on $\mathbb{N}^t$ defined by
 $\underline{n} \ge \underline{m}$ if and only if $n_i \ge m_i$ for all $1 \le i \le t$. Set
 $|\underline{n}| = n_1+\cdots+n_t$.
 
 If $R$ is an $\mathbb{N}^t$-graded ring and $L$ is an $\mathbb{N}^t$-graded $R$-module, then by $L_{\underline{n}}$,
 we always mean the $\underline{n}^{\rm th}$ graded component of $L$. By standard multigraded ring, we mean a
 multigraded ring which is generated in total degree one, i.e., $R$ is a standard $\mathbb{N}^t$-graded ring
 if $R = R_{\underline{0}}[R_{\underline{e}^1},\ldots,R_{\underline{e}^t}]$. All rings, graded or not, are assumed
 commutative with identity.
 
\section{Preliminaries}\label{Preliminaries}
 Let $A = A_0[x_1,\ldots,x_d]$ be a Noetherian standard $\mathbb{N}$-graded ring.
 Let $A_{+}$ be the ideal $\langle x_1,\ldots,x_d \rangle$ of $A$ generated by the elements of
 positive degree. Let $M$ be a finitely generated $\mathbb{N}$-graded $A$-module. For every integer $i \ge 0$,
 we denote the $i^{\rm th}$ local cohomology module of $M$ with respect to $A_{+}$ by $H_{A_{+}}^i(M)$.
 For every integer $i \ge 0$, we set
 \[ a_i(M) := \max\left\{ \mu : H_{A_{+}}^i(M)_{\mu} \neq 0 \right\} \]
 if $H_{A_{+}}^i(M) \neq 0$ and $a_i(M) := -\infty$ otherwise. The Castelnuovo-Mumford regularity of $M$ is defined by
 \[ \reg(M) := \max\left\{ a_i(M) + i : i \ge 0 \right\}.\]
 
 For a given short exact sequence of graded modules, by considering the corresponding long exact sequence of local
 cohomology modules, we can prove the following well-known result.  
 \begin{lemma}\label{lemma: properties of regularity}
  Let $A$ be as above. If $0 \rightarrow M_1 \rightarrow M_2 \rightarrow M_3 \rightarrow 0$ is a short exact
  sequence of finitely generated $\mathbb{N}$-graded $A$-modules, then we have the following.
  \begin{enumerate}
   \item[{\rm (i)}] $\reg(M_1) \le \max\{ \reg(M_2), \reg(M_3) + 1\}$.
   \item[{\rm (ii)}] $\reg(M_2) \le \max\{ \reg(M_1), \reg(M_3)\}$.
   \item[{\rm (iii)}] $\reg(M_3) \le \max\{ \reg(M_1) - 1, \reg(M_2)\}$.
  \end{enumerate}
 \end{lemma}
 We use the following well-known lemma to prove our main result inductively.
 \begin{lemma}\label{lemma: dimension reduction relation of regularity}
  Let $A$ be a Noetherian standard $\mathbb{N}$-graded ring and $M$ a finitely generated $\mathbb{N}$-graded
  $A$-module. Let $x$ be a homogeneous element in $A$ of positive degree $l$. Then we have the following inequality:
  \[\reg(M) \le \max\{\reg(0 :_M x), \reg(M/xM) - l + 1\}.\]
  Over polynomial rings over fields, if $x$ is such that $\dim(0 :_M x) \le 1$, then the inequality could be replaced
  by equality.
 \end{lemma}
 Now we give some preliminaries on $\mathbb{N}^t$-graded modules. We start with the following lemma.
 \begin{lemma}\label{lemma: Artin-Rees}
 Let $R = \bigoplus_{\underline{n}\in\mathbb{N}^t}R_{\underline{n}}$ be a Noetherian $\mathbb{N}^t$-graded ring, and let
 $L = \bigoplus_{\underline{n} \in \mathbb{N}^t}L_{\underline{n}}$ be a finitely generated $\mathbb{N}^t$-graded $R$-module.
 Set $A=R_{\underline{0}}$. Let $J$ be an ideal of $A$. Then there exists a positive integer $k$ such that
 \[J^m L_{\underline{n}}\cap H_J^0(L_{\underline{n}}) = 0 \quad\mbox{for all }~ \underline{n}\in\mathbb{N}^t
 \mbox{ and }~ m\ge k.\]
\end{lemma}
\begin{proof}
 Let $I=JR$ be the ideal of $R$ generated by $J$. Since $R$ is Noetherian and $L$ is a finitely generated $R$-module,
 then by Artin-Rees lemma, there exists a positive integer $c$ such that 
 \begin{align}
 (I^m L)\cap H_I^0(L) &=I^{m-c} \left((I^c L)\cap H_I^0(L)\right)\quad\mbox{for all }m\ge c\nonumber\\
                               &\subseteq I^{m-c} H_I^0(L) \quad\mbox{for all }m\ge c.
                               \label{lemma: Artin_Rees: equation 1}
 \end{align}
 Now consider the ascending chain of submodules of $L$:
 \[(0:_L I) \subseteq (0:_L I^2) \subseteq (0:_L I^3)\subseteq \cdots.\]
 Since $L$ is a Noetherian $R$-module, there exists some $l$ such that
 \begin{equation}\label{lemma: Artin_Rees: equation 2}
  (0:_L I^l) = (0:_L I^{l+1}) = (0:_L I^{l+2}) = \cdots = H_I^0(L).
 \end{equation}
 Set $k:= c+l$. Then from \eqref{lemma: Artin_Rees: equation 1} and \eqref{lemma: Artin_Rees: equation 2}, we have
 \[(I^m L)\cap H_I^0(L) \subseteq I^{m-c} (0 :_L I^{m-c}) = 0\quad\mbox{for all }m\ge k,\]
 which gives $(J^m L_{\underline{n}}) \cap H_J^0(L_{\underline{n}}) = 0$ for all $\underline{n} \in \mathbb{N}^t$
 and $m\ge k$.
\end{proof}
 Now we are aiming to obtain some invariant of multigraded module with the help of the following result.
 \begin{lemma}\label{lemma: annihilator stability for multigraded modules}
  Let $R$ be a Noetherian standard $\mathbb{N}^t$-graded ring and $L$ an $\mathbb{N}^t$-graded $R$-module finitely
  generated in degrees $\le \underline{u}$. Set $A=R_{\underline{0}}$. Then we have the following.
  \begin{enumerate}
   \item[\rm (i)] For any $\underline{v} \ge \underline{u}$, $\ann_A(L_{\underline{v}}) \subseteq
   \ann_A(L_{\underline{n}})$ for all $\underline{n} \ge \underline{v}$, and hence
   $\dim_A(L_{\underline{v}}) \ge \dim_A(L_{\underline{n}})$ for all $\underline{n} \ge \underline{v}$.
   \item[\rm (ii)] There exists $\underline{v} \in \mathbb{N}^t$ such that
    $\ann_A(L_{\underline{n}}) = \ann_A(L_{\underline{v}})$ for all $\underline{n} \ge \underline{v}$, and hence 
    $\dim_A(L_{\underline{n}}) = \dim_A(L_{\underline{v}})$ for all $\underline{n} \ge \underline{v}$.
  \end{enumerate}
 \end{lemma}
 \begin{proof}
  (i) Let $\underline{v} \ge \underline{u}$. Since $R$ is standard and $L$ is an $\mathbb{N}^t$-graded $R$-module
  finitely generated in degrees $\le \underline{u}$, for any $\underline{n} \ge \underline{v}$ ($\ge \underline{u}$), we have
  \[ L_{\underline{n}} = R_{\underline{e}^1}^{n_1 - v_1} R_{\underline{e}^2}^{n_2 - v_2} \cdots
     R_{\underline{e}^t}^{n_t - v_t} L_{\underline{v}},\]
  which gives $\ann_A(L_{\underline{v}}) \subseteq \ann_A(L_{\underline{n}})$, and hence
  $\dim_A(L_{\underline{v}}) \ge \dim_A(L_{\underline{n}})$ for all $\underline{n} \ge \underline{v}$.
  
  (ii) Consider $\mathcal{C} := \{\ann_A(L_{\underline{n}}) : \underline{n} \ge \underline{u}\}$, a collection of
  ideals of $A$. Since $A$ is Noetherian, $\mathcal{C}$ has a maximal element $\ann_A(L_{\underline{v}})$, say.
  Then by part (i), it follows that $\ann_A(L_{\underline{n}}) = \ann_A(L_{\underline{v}})$ for all
  $\underline{n} \ge \underline{v}$, and hence $\dim_A(L_{\underline{n}}) = \dim_A(L_{\underline{v}})$
  for all $\underline{n} \ge \underline{v}$.
 \end{proof}
 Let us introduce the following invariant of multigraded module on which we apply induction to prove our main result.
 \begin{definition}\label{definition: saturated dimension}
  Let $R$ be a Noetherian standard $\mathbb{N}^t$-graded ring and $L$ a finitely generated $\mathbb{N}^t$-graded
  $R$-module. We call $\underline{v} \in \mathbb{N}^t$ {\it an annihilator stable point of $L$} if 
  \[ \ann_{R_{\underline{0}}}(L_{\underline{n}}) = \ann_{R_{\underline{0}}}(L_{\underline{v}})\quad
     \mbox{for all }~\underline{n} \ge \underline{v}.\]
  In this case, we call $s := \dim_{R_{\underline{0}}}(L_{\underline{v}})$ as {\it the saturated dimension of $L$}.
 \end{definition}
 \begin{remark}\label{remark: saturated dimension existence}
  Existence of an annihilator stable point of $L$ (with the hypothesis given in the
  Definition~\ref{definition: saturated dimension}) follows from
  Lemma~\ref{lemma: annihilator stability for multigraded modules}(ii). Let $\underline{v}, \underline{w} \in
  \mathbb{N}^t$ be two annihilator stable points of $L$, i.e., 
  \begin{align*}
   &\ann_{R_{\underline{0}}}(L_{\underline{n}}) = \ann_{R_{\underline{0}}}(L_{\underline{v}})\quad
     \mbox{for all }~\underline{n} \ge \underline{v} \\
   \mbox{and}\quad &\ann_{R_{\underline{0}}}(L_{\underline{n}}) = \ann_{R_{\underline{0}}}(L_{\underline{w}})
   \quad \mbox{for all }~\underline{n} \ge \underline{w}.
  \end{align*}
  If we denote $\dim_{R_{\underline{0}}}(L_{\underline{v}})$ and $\dim_{R_{\underline{0}}}(L_{\underline{w}})$ by
  $s(\underline{v})$ and $s(\underline{w})$ respectively, then observe that $s(\underline{v}) = s(\underline{w})$. Thus
  the saturated dimension of $L$ is well-defined.
 \end{remark}
 Let us recall the following result from \cite[Lemma~3.3]{We04}.
 \begin{lemma}\label{lemma: fixing one component of gradings}
  Let $R$ be a Noetherian standard $\mathbb{N}^t$-graded ring and $L$ a finitely generated $\mathbb{N}^t$-graded
  $R$-module. For any fixed integers $1 \le i \le t$ and $\lambda \in \mathbb{N}$, set
  \[ S_i := \bigoplus_{\{\underline{n}\in\mathbb{N}^t : n_i = 0\}} R_{\underline{n}}\quad\quad\mbox{ and }\quad
     M_{i\lambda} := \bigoplus_{\{\underline{n}\in\mathbb{N}^t : n_i = \lambda \}} L_{\underline{n}}.
  \]
  Then $S_i$ is a Noetherian standard $\mathbb{N}^{t-1}$-graded ring and $M_{i\lambda}$ is a finitely
  generated $\mathbb{N}^{t-1}$-graded $S_i$-module.
 \end{lemma}
 \begin{discussion}\label{discussion}
  Let
  \[ 
    R = \bigoplus_{(\underline{n},i)\in \mathbb{N}^{t+1}} R_{(\underline{n},i)} \quad\quad\mbox{and}\quad
    L = \bigoplus_{(\underline{n},i)\in \mathbb{N}^{t+1}} L_{(\underline{n},i)}
  \]
  be a Noetherian $\mathbb{N}^{t+1}$-graded ring and a finitely generated $\mathbb{N}^{t+1}$-graded
  $R$-module respectively. For each $\underline{n} \in \mathbb{N}^t$, we set
  \[ 
    R_{(\underline{n},\star)} := \bigoplus_{i \in \mathbb{N}} R_{(\underline{n},i)} \quad\quad\mbox{and}\quad
    L_{(\underline{n},\star)} := \bigoplus_{i \in \mathbb{N}} L_{(\underline{n},i)}.
  \]
  We give $\mathbb{N}^t$-grading structures on
  \[ R = \bigoplus_{\underline{n} \in \mathbb{N}^t} R_{(\underline{n},\star)} \quad\quad\mbox{and}\quad
     L = \bigoplus_{\underline{n} \in \mathbb{N}^t} L_{(\underline{n},\star)}
   \]
  in the obvious way, i.e., by setting $R_{(\underline{n},\star)}$ and $L_{(\underline{n},\star)}$ as the
  $\underline{n}^{\rm th}$ graded components of $R$ and $L$ respectively. Then clearly, for any
  $\underline{m}, \underline{n} \in \mathbb{N}^t$, we have
  \[
   R_{(\underline{m},\star)} \cdot R_{(\underline{n},\star)} \subseteq R_{(\underline{m} + \underline{n},\star)}
   \quad\mbox{and}\quad 
   R_{(\underline{m},\star)} \cdot L_{(\underline{n},\star)} \subseteq L_{(\underline{m} + \underline{n},\star)}.
  \]
  Thus $R$ is an $\mathbb{N}^t$-graded ring and $L$ is an $\mathbb{N}^t$-graded $R$-module. Since we are changing
  only the grading, $R$ is anyway Noetherian. Since $L$ is finitely generated $\mathbb{N}^{t+1}$-graded $R$-module,
  it is just an observation that $L = \bigoplus_{\underline{n} \in \mathbb{N}^t} L_{(\underline{n},\star)}$ is
  finitely generated as $\mathbb{N}^t$-graded $R$-module. Now we set $A := R_{(\underline{0},\star)}$. Note that
  $A$ is a Noetherian $\mathbb{N}$-graded ring, and for each $\underline{n} \in \mathbb{N}^t$,
  $R_{(\underline{n},\star)}$ and $L_{(\underline{n},\star)}$ are finitely generated $\mathbb{N}$-graded
  $A$-modules.
 \end{discussion}
  We are going to refer the following hypothesis repeatedly in the rest of the paper.
 \begin{hypothesis}\label{hypothesis}
  Let
  \[ R = \bigoplus_{(\underline{n},i)\in \mathbb{N}^{t+1}} R_{(\underline{n},i)} \]
  be a Noetherian $\mathbb{N}^{t+1}$-graded ring, {\it which need not be standard}. Let
  \[ L = \bigoplus_{(\underline{n},i)\in \mathbb{N}^{t+1}} L_{(\underline{n},i)} \]
  be a finitely generated $\mathbb{N}^{t+1}$-graded $R$-module. For each $\underline{n} \in \mathbb{N}^t$, we set
  \[ 
     R_{(\underline{n},\star)} := \bigoplus_{i \in \mathbb{N}} R_{(\underline{n},i)} \quad\quad\mbox{and}\quad
     L_{(\underline{n},\star)} := \bigoplus_{i \in \mathbb{N}} L_{(\underline{n},i)}.
  \]
  Also set $A := R_{(\underline{0},\star)}$. Suppose
  $R = \bigoplus_{\underline{n} \in \mathbb{N}^t} R_{(\underline{n},\star)}$ and $A = R_{(\underline{0},\star)}$
  are standard as $\mathbb{N}^t$-graded ring and $\mathbb{N}$-graded ring respectively, i.e.
  \[
    R = R_{(\underline{0},\star)} [R_{(\underline{e}^1,\star)}, R_{(\underline{e}^2,\star)}, \ldots,
                                   R_{(\underline{e}^t,\star)}]
    \quad\mbox{and}\quad R_{(\underline{0},\star)} = R_{(\underline{0},0)}[R_{(\underline{0},1)}].
  \]
  Assume $A_0 = R_{(\underline{0},0)}$ is Artinian local with the
  maximal ideal $\mathfrak{m}$. Since $A$ is a Noetherian standard $\mathbb{N}$-graded ring, we assume
  $A = A_0[x_1,\ldots,x_d]$ for some $x_1,\ldots,x_d \in A_1$. Let $A_{+} = \langle x_1,\ldots,x_d \rangle$.
 \end{hypothesis} 
  With the Hypothesis~\ref{hypothesis}, from Discussion~\ref{discussion}, we have the following.
  \begin{enumerate}
   \item[(0)] $R = \bigoplus_{(\underline{n},i)\in \mathbb{N}^{t+1}} R_{(\underline{n},i)}$ is not necessarily
   standard as $\mathbb{N}^{t+1}$-graded ring.
   \item[(1)] $R = \bigoplus_{\underline{n} \in \mathbb{N}^t} R_{(\underline{n},\star)}$ is a Noetherian standard
   $\mathbb{N}^t$-graded ring.
   \item[(2)] $A = R_{(\underline{0},\star)}$ is a Noetherian standard $\mathbb{N}$-graded ring.
   \item[(3)] $L = \bigoplus_{\underline{n} \in \mathbb{N}^t} L_{(\underline{n},\star)}$ is a finitely generated
   $\mathbb{N}^t$-graded $R$-module.
   \item[(4)] For each $\underline{n} \in \mathbb{N}^t$, $R_{(\underline{n},\star)}$ and
   $L_{(\underline{n},\star)}$ are finitely generated $\mathbb{N}$-graded $A$-modules.
  \end{enumerate}

  Here is an example which satisfies the Hypothesis~\ref{hypothesis}.
  \begin{example}\label{example satisfying the hypothesis}
   Let $A$ be a Noetherian standard $\mathbb{N}$-graded algebra over an Artinian local ring $A_0$. Let
  $I_1,\ldots,I_t$ be homogeneous ideals of $A$ and $M$ a finitely generated $\mathbb{N}$-graded $A$-module.
  Let $R = A[I_1 T_1,\ldots, I_t T_t]$ be the Rees algebra of $I_1,\ldots,I_t$ over the graded ring $A$ and let
  $L = M[I_1 T_1,\ldots, I_t T_t]$ be the Rees module of $M$ with respect to the ideals $I_1,\ldots,I_t$.
  We give $\mathbb{N}^{t+1}$-grading structures on $R$ and $L$ by setting $(\underline{n},i)^{\rm th}$ graded
  components of $R$ and $L$ as the $i^{\rm th}$ graded components of the $\mathbb{N}$-graded $A$-modules
  $I_1^{n_1}\cdots I_t^{n_t} A$ and $I_1^{n_1}\cdots I_t^{n_t} M$ respectively. Then clearly, $R$ is
  a Noetherian $\mathbb{N}^{t+1}$-graded ring and $L$ is a finitely generated $\mathbb{N}^{t+1}$-graded $R$-module.
  Note that $R$ is not necessarily standard as $\mathbb{N}^{t+1}$-graded ring. Also note that for each
  $\underline{n} \in \mathbb{N}^t$,
  \begin{align*}
   R_{(\underline{n},\star)}
   &= \bigoplus_{i \in \mathbb{N}} R_{(\underline{n},i)} = I_1^{n_1}\cdots I_t^{n_t} A \\
   \mbox{and }\quad L_{(\underline{n},\star)}
   &= \bigoplus_{i \in \mathbb{N}} L_{(\underline{n},i)} = I_1^{n_1}\cdots I_t^{n_t} M.
  \end{align*}
  Since $R = A[I_1 T_1,\ldots, I_t T_t] =
  R_{(\underline{0},\star)}[R_{(\underline{e}^1,\star)},\ldots,R_{(\underline{e}^t,\star)}]$,
  $R = \bigoplus_{\underline{n}\in \mathbb{N}^t}R_{(\underline{n},\star)}$ is standard as $\mathbb{N}^t$-graded ring.
  Thus $R$ and $L$ are satisfying the Hypothesis~\ref{hypothesis}.
  \end{example}
  From now onwards, by $R$ and $L$, we mean $\mathbb{N}^t$-graded ring
  $\bigoplus_{\underline{n} \in \mathbb{N}^t} R_{(\underline{n},\star)}$ and $\mathbb{N}^t$-graded $R$-module
  $\bigoplus_{\underline{n} \in \mathbb{N}^t} L_{(\underline{n},\star)}$
  (satisfying the Hypothesis~\ref{hypothesis}) respectively.
  
 \section{Linear bounds of regularity}\label{Linear bounds of regularity}
 In this section, we are aiming to prove that the regularity of $L_{(\underline{n},\star)}$ as an $\mathbb{N}$-graded
 $A$-module is bounded by a linear function of $\underline{n}$ by using induction on the saturated dimension of the
 $\mathbb{N}^t$-graded $R$-module $L$. Here is the base case.
 \begin{theorem}\label{theorem: bounds of regularity for saturated dimension 0}
  With the {\rm Hypothesis~\ref{hypothesis}}, let $L = \bigoplus_{\underline{n} \in \mathbb{N}^t} L_{(\underline{n},\star)}$
  be generated in degrees $\le \underline{u}$. If $\dim_A(L_{(\underline{v},\star)}) = 0$ for some
  $\underline{v} \ge \underline{u}$, then there exists an integer $k$ such that
  \[ \reg(L_{(\underline{n},\star)}) < |\underline{n} - \underline{u}| k + k \quad
     \mbox{for all }~\underline{n} \ge \underline{v}.\]
 \end{theorem}
 \begin{proof}
  Let $\dim_A(L_{(\underline{v},\star)}) = 0$ for some $\underline{v} \ge \underline{u}$. Then from
  Lemma~\ref{lemma: annihilator stability for multigraded modules}(i), we have
  \[ \dim_A(L_{(\underline{n},\star)}) = 0 \quad\mbox{ for all }~\underline{n} \ge \underline{v}.\]
  By Grothendieck vanishing theorem, we have
  \[ H_{A_{+}}^i(L_{(\underline{n},\star)}) = 0 \quad\mbox{ for all }~i > 0\mbox{ and }~\underline{n} \ge \underline{v}.\]
  Therefore in this case
  \begin{equation}\label{theorem: bounds of regularity for saturated dimension 0: equation 1}
   \reg(L_{(\underline{n},\star)}) = \max\left\{\mu : H_{A_{+}}^0(L_{(\underline{n},\star)})_{\mu} \neq 0\right\}
   \quad\mbox{ for all }~\underline{n} \ge \underline{v}.
  \end{equation}
  
  Now consider the finite collection
  \[
    \mathcal{D} := \{R_{(\underline{e}^1,\star)},R_{(\underline{e}^2,\star)},\ldots,R_{(\underline{e}^t,\star)},
    L_{(\underline{u},\star)}\}.
  \]
  Since each member of $\mathcal{D}$ is finitely generated $\mathbb{N}$-graded $A=A_0[x_1,\ldots,x_d]$-module, we may assume
  that every member of $\mathcal{D}$ is generated in degrees $\le k_1$ for some $k_1 \in \mathbb{N}$.
  Since $L$ is a finitely generated $\mathbb{N}^t$-graded $R$-module and $A_{+}$ is an ideal of
  $A$ ($= R_{(\underline{0},\star)}$), by Lemma~\ref{lemma: Artin-Rees},
  there exists a positive integer $k_2$ such that
  \begin{equation}\label{theorem: bounds of regularity for saturated dimension 0: equation 2}
   (A_{+})^{k_2} L_{(\underline{n},\star)} \cap H_{A_{+}}^0(L_{(\underline{n},\star)}) = 0 \quad
   \mbox{ for all }~\underline{n} \in \mathbb{N}^t.
  \end{equation}
  Now set $k := k_1 + k_2$. We claim that
  \begin{equation}\label{theorem: bounds of regularity for saturated dimension 0: equation 3}
   H_{A_{+}}^0(L_{(\underline{n},\star)})_{\mu} = 0\quad\mbox{for all }~\underline{n} \ge \underline{v}
   ~\mbox{ and }~\mu \ge |\underline{n} - \underline{u}| k + k.
  \end{equation}
  
  To show \eqref{theorem: bounds of regularity for saturated dimension 0: equation 3},
  fix $\underline{n} \ge \underline{v}$ and $\mu \ge |\underline{n} - \underline{u}| k + k$. Assume
  $X \in H_{A_{+}}^0(L_{(\underline{n},\star)})_{\mu}$. Note that
  the homogeneous (with respect to $\mathbb{N}$-grading over $A$) element $X$ of
  \[
  L_{(\underline{n},\star)} =
  R_{(\underline{e}^1,\star)}^{n_1 - u_1} R_{(\underline{e}^2,\star)}^{n_2 - u_2} \cdots
  R_{(\underline{e}^t,\star)}^{n_t - u_t} L_{(\underline{u},\star)}
  \]
  can be written as a finite sum of elements of the following type:
  \[(r_{11}r_{12}\cdots r_{1~n_1-u_1}) (r_{21}r_{22}\cdots r_{2~n_2-u_2})\cdots (r_{t1}r_{t2}\cdots r_{t~n_t-u_t}) Y\]
  for some homogeneous (with respect to $\mathbb{N}$-grading over $A$) elements
  \[
  r_{i1}, r_{i2},\ldots, r_{i~n_i - u_i} \in R_{(\underline{e}^i,\star)}\quad
  \mbox{for all }~ 1 \le i \le t,\mbox{ and }~Y \in L_{(\underline{u},\star)}.
  \]
  Considering the homogeneous degree with respect to $\mathbb{N}$-grading over $A$, we have
  \[ \deg(Y) + \sum_{i=1}^t\left\{\deg(r_{i1})+\deg(r_{i2})+\cdots+\deg(r_{i~n_i-u_i})\right\} = \mu
     \ge |\underline{n} - \underline{u}| k + k,
  \]
  which gives at least one of the elements
  \[r_{11},r_{12},\ldots,r_{1~n_1-u_1},\ldots,r_{t1},r_{t2},\ldots,r_{t~n_t-u_t}\quad\mbox{and }~Y\]
  is of degree $\ge k$. In first case, we consider $\deg(r_{ij}) \ge k$ for some $i,j$. Since
  $R_{(\underline{e}^i,\star)}$ is an $\mathbb{N}$-graded $A$-module generated in degrees $\le k_1$, we have
  \begin{align*}
   r_{ij} \in \left(R_{(\underline{e}^i,\star)}\right)_{\deg(r_{ij})}
   &= (A_1)^{\deg(r_{ij}) - k_1} \left(R_{(\underline{e}^i,\star)}\right)_{k_1}\\
   &\subseteq (A_{+})^{k_2} R_{(\underline{e}^i,\star)}\quad \mbox{[as $\deg(r_{ij})-k_1 \ge k-k_1 = k_2$]}.
  \end{align*}
  In another case, we consider $\deg(Y) \ge k$. In this case also, since $L_{(\underline{u},\star)}$ is an
  $\mathbb{N}$-graded $A$-module generated in degrees $\le k_1$, we have
  \begin{align*}
   Y \in \left(L_{(\underline{u},\star)}\right)_{\deg(Y)}
   &= (A_1)^{\deg(Y) - k_1} \left(L_{(\underline{u},\star)}\right)_{k_1}\\
   &\subseteq (A_{+})^{k_2} L_{(\underline{u},\star)}\quad \mbox{[as $\deg(Y)-k_1 \ge k-k_1 = k_2$]}.
  \end{align*}
  In both cases, the typical element
  $(r_{11}r_{12}\cdots r_{1~n_1-u_1}) \cdots (r_{t1}r_{t2}\cdots r_{t~n_t-u_t}) Y$ is in
  \[ 
   (A_{+})^{k_2} R_{(\underline{e}^1,\star)}^{n_1-u_1} R_{(\underline{e}^2,\star)}^{n_2-u_2} \cdots
   R_{(\underline{e}^t,\star)}^{n_t-u_t} L_{(\underline{u},\star)} = (A_{+})^{k_2} L_{(\underline{n},\star)},
  \]
  and hence $X \in (A_{+})^{k_2} L_{(\underline{n},\star)}$. Therefore
  $X \in (A_{+})^{k_2} L_{(\underline{n},\star)} \cap H_{A_{+}}^0(L_{(\underline{n},\star)})$,
  which gives $X = 0$ by \eqref{theorem: bounds of regularity for saturated dimension 0: equation 2}. Thus we have
  \[
   H_{A_{+}}^0(L_{(\underline{n},\star)})_{\mu} = 0\quad\mbox{for all }~\underline{n} \ge \underline{v}
   ~\mbox{ and }~\mu \ge |\underline{n} - \underline{u}| k + k,
  \]
  and hence the theorem follows from \eqref{theorem: bounds of regularity for saturated dimension 0: equation 1}.
 \end{proof}
 Now we give the inductive step to prove the following linear boundedness result.
 \begin{theorem}\label{theorem: bounds of regularity for saturated dimension positive}
  With the {\rm Hypothesis~\ref{hypothesis}}, there exist $\underline{u} \in \mathbb{N}^t$ and an integer $k$ such that
  \[ \reg(L_{(\underline{n},\star)}) < |\underline{n}| k + k \quad\mbox{for all }~\underline{n} \ge \underline{u}.\]
  In particular, if $t = 1$, then there exist two integers $k, k'$ such that
  \[ \reg(L_{(n,\star)}) \le n k + k' \quad\mbox{for all }~n \in \mathbb{N}.\]
 \end{theorem}
 \begin{proof}
  Let $\underline{v} \in \mathbb{N}^t$ be an annihilator stable point of $L$ and $s$ the saturated dimension of $L$.
  Without loss of generality, we may assume that $L$ is finitely generated as $R$-module in degrees
  $\le \underline{v}$. We prove the theorem by induction on $s$. If $s = 0$, then the theorem follows from
  Theorem~\ref{theorem: bounds of regularity for saturated dimension 0} by taking $\underline{u} := \underline{v}$.
  Therefore we may as well assume that $s > 0$ and the theorem holds true for all such finitely generated
  $\mathbb{N}^t$-graded $R$-modules with the saturated dimensions $\le s - 1$.
  
  Let $\mathfrak{n} := \mathfrak{m} \oplus A_{+}$ be the maximal homogeneous ideal of $A$. We claim that
  \[ \mathfrak{n} \notin \Min\left(A/\ann_A(L_{(\underline{v},\star)})\right). \]
  Since the collection of all minimal prime ideals of $A$ containing $\ann_A(L_{(\underline{v},\star)})$ are
  associated prime ideals of $A/\ann_A(L_{(\underline{v},\star)})$, they are homogeneous, and hence they must be
  contained in $\mathfrak{n}$. Thus if the above claim is not true, then we have
  \[ \Min\left(A/\ann_A(L_{(\underline{v},\star)})\right) = \{\mathfrak{n}\},\]
  and hence $s = \dim_A(L_{(\underline{v},\star)}) = 0$, which is a contradiction. Therefore the above claim is true,
  and hence by prime avoidance lemma and using the fact that $\mathfrak{n}$ is the only homogeneous prime ideal
  of $A$ containing $A_{+}$ (as $(A_0,\mathfrak{m})$ is Artinian local), we have
  \[ A_{+} \nsubseteq \bigcup\left\{P:P \in \Min\left(A/\ann_A(L_{(\underline{v},\star)})\right)\right\}.\]
  Then by graded version of prime avoidance lemma, we may choose a homogeneous element $x$ in $A$ of positive
  degree such that
  \[ x \notin \bigcup\left\{P:P \in \Min\left(A/\ann_A(L_{(\underline{v},\star)})\right)\right\}. \]
  Note that $\ann_A(L_{(\underline{n},\star)}) = \ann_A(L_{(\underline{v},\star)})$ for all
  $\underline{n} \ge \underline{v}$. Therefore for all $\underline{n} \ge \underline{v}$, we have
  \begin{align*}
   &\dim_A\left(L_{(\underline{n},\star)}/xL_{(\underline{n},\star)}\right),~
    \dim_A\big(0 :_{L_{(\underline{n},\star)}} x\big) \le \dim_A(L_{(\underline{n},\star)}) - 1 = s - 1 \\
   &\mbox{as }\quad \ann_A\left(L_{(\underline{n},\star)}/xL_{(\underline{n},\star)}\right),~
    \ann_A\big(0 :_{L_{(\underline{n},\star)}} x\big)\supseteq \langle \ann_A(L_{(\underline{n},\star)}), x \rangle.
  \end{align*}
  
  Now observe that $L/xL$ and $(0 :_L x)$ are finitely generated $\mathbb{N}^t$-graded $R$-modules with saturated
  dimensions $\le s - 1$. Therefore, by induction hypothesis, there exist $\underline{w}$ and $\underline{w}'$ in
  $\mathbb{N}^t$ and two integers $k_1, k_2$ such that 
  \begin{align*}
   \reg\left(L_{(\underline{n},\star)}/xL_{(\underline{n},\star)}\right)
   &< |\underline{n}| k_1 + k_1\quad\mbox{for all }\underline{n} \ge \underline{w} \\
   \mbox{and}\quad\reg\big(0 :_{L_{(\underline{n},\star)}} x\big)
   &< |\underline{n}| k_2 + k_2\quad\mbox{for all }\underline{n} \ge \underline{w}'.
  \end{align*}
  Set $k := \max\{k_1, k_2\}$ and $\underline{u} := \max\{\underline{w}, \underline{w}'\}$
  (i.e., $u_i := \max\{w_i,w'_i\}$ for all $1 \le i \le t$, and $\underline{u} := (u_1,\ldots,u_t)$). Then from
  Lemma~\ref{lemma: dimension reduction relation of regularity}, we have
  \begin{align*}
   \reg(L_{(\underline{n},\star)}) &\le
   \max\left\{\reg\left(L_{(\underline{n},\star)}/xL_{(\underline{n},\star)}\right), \reg\big(0 :_{L_{(\underline{n},\star)}} x\big)\right\}\\
   & < |\underline{n}| k + k \quad \mbox{for all }\underline{n} \ge \underline{u}.
  \end{align*}
  This completes the proof of the first part of the theorem. To prove the second part, assume $t = 1$.
  Then from the first part, there exist $u \in \mathbb{N}$ and an integer $k$ such that
  \[ \reg(L_{(n,\star)}) < n k + k \quad\mbox{for all }~n \ge u.\]
  Set $k' := \max\left\{k, \reg(L_{(n,\star)}) : 0 \le n \le u-1\right\}$. Then clearly, we have
  \[ \reg(L_{(n,\star)}) \le n k + k' \quad\mbox{for all }~n \in \mathbb{N},\]
  which completes the proof of the theorem.
 \end{proof}
 Above theorem gives the result that $\reg(L_{(\underline{n},\star)})$ has linear bound for all
 $\underline{n} \ge \underline{u}$, for some $\underline{u} \in \mathbb{N}^t$. Now we prove the result
 for all $\underline{n} \in \mathbb{N}^t$.
 \begin{theorem}\label{theorem: bounds of regularity for multigraded module}
  With the {\rm Hypothesis~\ref{hypothesis}}, there exist two integers $k, k'$ such that
  \[ \reg\left(L_{(\underline{n},\star)}\right) \le (n_1 + \cdots + n_t) k + k' \quad
  \mbox{for all }~\underline{n} \in \mathbb{N}^t.\]
 \end{theorem}
 \begin{proof}
  We prove the theorem by induction on $t$. If $t = 1$, then the theorem follows from the second part of the
  Theorem~\ref{theorem: bounds of regularity for saturated dimension positive}. Therefore we may as well assume that
  $t \ge 2$ and the theorem holds true for $t-1$.
  
  By Theorem~\ref{theorem: bounds of regularity for saturated dimension positive}, there exist
  $\underline{u} \in \mathbb{N}^t$ and an integer $k_1$ such that
  \begin{equation}\label{theorem: bounds of regularity for multigraded module: equation 1}
   \reg\left(L_{(\underline{n},\star)}\right) < |\underline{n}| k_1 + k_1 \quad\mbox{for all }~\underline{n} \ge \underline{u}.
  \end{equation}
  Now for each $1 \le i \le t$ and $0 \le \lambda < u_i$, we set
  \[
   S_i := \bigoplus_{\{\underline{n} \in \mathbb{N}^t: n_i = 0\}} R_{(\underline{n},\star)}\quad\quad
   \mbox{and}\quad M_{i\lambda} := \bigoplus_{\{\underline{n} \in \mathbb{N}^t: n_i
                                 = \lambda\}} L_{(\underline{n},\star)}.
  \]
  Then from Lemma~\ref{lemma: fixing one component of gradings}, $S_i$ is a Noetherian standard
  $\mathbb{N}^{t-1}$-graded ring and $M_{i\lambda}$ is a finitely generated $\mathbb{N}^{t-1}$-graded $S_i$-module.
  Therefore, by induction hypothesis, for each $1 \le i \le t$ and $0 \le \lambda < u_i$, there exist two integers
  $k_{i\lambda}$ and $k'_{i\lambda}$ such that
  \begin{align}\label{theorem: bounds of regularity for multigraded module: equation 2}
   \reg\left(L_{(n_1,\ldots,n_{i-1},\lambda,n_{i+1},\ldots,n_t,\star)}\right)
   &\le (n_1+\cdots+n_{i-1}+n_{i+1}+\cdots+n_t) k_{i\lambda} + k'_{i\lambda}                      \nonumber \\
   &= (n_1+\cdots+n_{i-1}+\lambda+n_{i+1}+\cdots+n_t) k_{i\lambda} + k''_{i\lambda}                          \\
   &\quad\quad\quad\quad\quad~~\mbox{ for all }n_1,\ldots,n_{i-1},n_{i+1},\ldots,n_t \in \mathbb{N},\nonumber
  \end{align}
  where $k''_{i\lambda} = k'_{i\lambda} - \lambda k_{i\lambda}$. Now set
  \begin{align*}
   k  &:= \max\left\{k_1, k_{i\lambda} : 1 \le i \le t, 0 \le \lambda < u_i\right\}\quad\mbox{and}\\
   k' &:= \max\left\{k_1, k''_{i\lambda} : 1 \le i \le t, 0 \le \lambda < u_i\right\}.
  \end{align*}
  We claim that
  \begin{equation}\label{theorem: bounds of regularity for multigraded module: equation 3}
   \reg\left(L_{(\underline{n},\star)}\right) \le |\underline{n}| k + k' \quad\mbox{for all }~\underline{n} \in \mathbb{N}^t.
  \end{equation}
  To prove \eqref{theorem: bounds of regularity for multigraded module: equation 3}, consider an arbitrary
  $\underline{n} \in \mathbb{N}^t$. If $\underline{n} \ge \underline{u}$, then
  \eqref{theorem: bounds of regularity for multigraded module: equation 3} follows from
  \eqref{theorem: bounds of regularity for multigraded module: equation 1}. Otherwise if
  $\underline{n} \ngeqslant \underline{u}$, then we have $n_i < u_i$ for at least one $i \in \{1,\ldots,t\}$,
  and hence in this case, \eqref{theorem: bounds of regularity for multigraded module: equation 3}
  holds true by \eqref{theorem: bounds of regularity for multigraded module: equation 2}.
 \end{proof}
 Now we have arrived at the main goal of this article.
 \begin{corollary}\label{corollary: bounds of regularity of ideals power times module}
  Let $A$ be a Noetherian standard $\mathbb{N}$-graded algebra over an Artinian local ring $A_0$. Let
  $I_1,\ldots,I_t$ be homogeneous ideals of $A$ and $M$ a finitely generated $\mathbb{N}$-graded $A$-module.
  Then there exist two integers $k, k'$ such that
  \[
    \reg(I_1^{n_1}\cdots I_t^{n_t} M) \le (n_1 + \cdots + n_t) k + k'
    \quad\mbox{for all }~n_1,\ldots,n_t \in \mathbb{N}.
  \]
 \end{corollary}
 \begin{proof}
  Let $R = A[I_1 T_1,\ldots, I_t T_t]$ be the Rees algebra of $I_1,\ldots,I_t$ over the graded ring $A$ and let
  $L = M[I_1 T_1,\ldots, I_t T_t]$ be the Rees module of $M$ with respect to the ideals $I_1,\ldots,I_t$. We give
  $\mathbb{N}^{t+1}$-grading structures on $R$ and $L$ by setting $(\underline{n},i)^{\rm th}$ graded
  components of $R$ and $L$ as the $i^{\rm th}$ graded components of the $\mathbb{N}$-graded $A$-modules
  $I_1^{n_1}\cdots I_t^{n_t} A$ and $I_1^{n_1}\cdots I_t^{n_t} M$ respectively.
  From Example~\ref{example satisfying the hypothesis}, note that $R$ and $L$ are satisfying the
  Hypothesis~\ref{hypothesis}, and in this case
  \[ L_{(\underline{n},\star)} = I_1^{n_1}\cdots I_t^{n_t} M \quad\mbox{ for all }~\underline{n} \in \mathbb{N}^t.\]
  Therefore the corollary follows from Theorem~\ref{theorem: bounds of regularity for multigraded module}.
 \end{proof}

 \section*{Acknowledgements}
 I would like to express my sincere gratitude to my supervisor, Prof. Tony J. Puthenpurakal, for his generous guidance and
 valuable suggestions concerning this article. I would like to thank the referee for pertinent comments. Finally,
 I thank NBHM, DAE, Govt. of India for providing financial support for this study.

\end{document}